\documentclass[12pt]{article}
\usepackage{amsmath,amsfonts,amssymb,amsthm}
\usepackage{bbm}
\usepackage[ansinew]{inputenc}
\usepackage{epsf,epsfig}
\usepackage{color}
\usepackage{lineno}
\DeclareMathOperator{\dist}{dist}

\DeclareMathOperator{\id}{id}

\DeclareMathOperator{\co}{co}
\DeclareMathOperator{\Proj}{Proj}

\DeclareMathOperator{\interior}{int}
\DeclareMathOperator{\chain}{chain}

\begin{document}

\newtheorem{oberklasse}{OberKlasse}
\newtheorem{lemma}[oberklasse]{Lemma}
\newtheorem{proposition}[oberklasse]{Proposition}
\newtheorem{theorem}[oberklasse]{Theorem}
\newtheorem{remark}[oberklasse]{Remark}
\newtheorem{corollary}[oberklasse]{Corollary}
\newtheorem{definition}[oberklasse]{Definition}
\newtheorem{example}[oberklasse]{Example}
\newtheorem{observation}[oberklasse]{Observation}
\newtheorem{assumption}{Assumption}
\newcommand{\clconvhull}{\ensuremath{\overline{\co}}}
\newcommand{\R}{\ensuremath{\mathbbm{R}}}
\newcommand{\N}{\ensuremath{\mathbbm{N}}}
\newcommand{\Z}{\ensuremath{\mathbbm{Z}}}
\newcommand{\Q}{\ensuremath{\mathbbm{Q}}}
\newcommand{\X}{\ensuremath{\mathcal{X}}}
\newcommand{\M}{\ensuremath{\mathcal{M}}}
\newcommand{\F}{\ensuremath{\mathcal{F}}}
\newcommand{\V}{\ensuremath{\mathcal{V}}}
\newcommand{\h}{\ensuremath{\mathcal{H}}}
\newcommand{\ClSets}{\ensuremath{\mathcal{A}}}
\newcommand{\OpSets}{\ensuremath{\mathcal{O}}}
\newcommand{\CpSets}{\ensuremath{\mathcal{C}}}
\newcommand{\CoCpSets}{\ensuremath{\mathcal{CC}}}
\newcommand{\powerset}{\ensuremath{\mathcal{P}}}
\newcommand{\CB}{\ensuremath{\mathcal{CBC}}}

\newcommand{\tc}{\textcolor}
\newcommand{\mc}{\mathcal}
\newcommand{\myinclude}[1]{}

\renewcommand{\phi}{\ensuremath{\varphi}}
\renewcommand{\epsilon}{\ensuremath{\varepsilon}}

\title{Robust boundary tracking for reachable sets of nonlinear differential inclusions}
\author{Janosch Rieger\footnote{Institut f\"ur Mathematik, Universit\"at Frankfurt,
Postfach 111932, D-60054 Frankfurt a.M., Germany; phone: +49/69/798-22715, fax: +49/69/798-28846.}}
\date{\today}
\maketitle

\begin{abstract}
The Euler scheme is up to date the most important numerical method for ordinary differential inclusions,
because the use of the available higher-order methods is prohibited by their enormous complexity after spatial discretization.
Therefore, it makes sense to reassess the Euler scheme and optimize its performance. In the present paper,
a considerable reduction of the computational cost is achieved by setting up a numerical method that 
computes the boundaries instead of the complete reachable sets of the fully discretized Euler scheme 
from lower-dimensional data only. Rigorous proofs for the propriety of this method are given, and numerical
examples illustrate the gain of computational efficiency as well as the robustness of the scheme against 
changes of topology of the reachable sets.
\end{abstract}


\section{Introduction}

Consider the differential inclusion
\begin{equation} \label{odi}
x'(t) \in F(t,x(t)),\quad x(0)\in X_0,
\end{equation}
where $F:\R\times\R^d\rightrightarrows\R^d$ is a Lipschitz continuous multivalued mapping that assigns to each pair $(t,x)\in\R\times\R^d$
a convex and compact subset of $\R^d$. Differential inclusions model global aspects of control and deterministic uncertainty. 
In a sense, they are the deterministic counterpart of stochastic differential equations. 
The reachable sets
\[\mc{R}(T,X_0)=\{x(T): x(\cdot)\ \text{solves \eqref{odi}}\}\]
are of considerable interest, because they are the sets of all states that 
can be reached using admissible controls or perturbations.
The theory of differential inclusions is well-developed, and the reader is referred to the monographs
\cite{Aubin:Cellina:84} and \cite{Deimling:92} for details. 

\medskip

The numerical approximation of the set of all solutions of \eqref{odi} and the corresponding reachable sets
remains a challenging subject even in the low-dimensional context, because 
it is massively affected by the curse of dimensionality.
This paper is concerned with the multivalued Euler scheme
\[\Phi:\R\times\R^d\rightrightarrows\R^d,\quad \Phi(x)=x+hF(t,x)\]
for differential inclusions of form \eqref{odi} that has been investigated in \cite{Dontchev:Farkhi:89} and several papers since.
The reachable set $\mc{R}(T,X_0)$ can be approximated by fixing some $N\in\N$ and $h>0$ such that $T=Nh$ and computing
\[\mc{R}_h(t_{n+1},X_0) = \Phi(t_n,\mc{R}_h(t_n,X_0)),\ n=0,\ldots,N-1,\quad \mc{R}_h(t_0,X_0)=X_0\]
recursively, where $t_n=nh$. For practical computations, it is necessary to discretize these sets. This is usually done by 
introducing a grid $\Delta_\rho=\rho\Z^d$ and a fully discrete Euler scheme 
\begin{eqnarray*}
&\tilde\Phi_\alpha:\R\times\Delta_\rho\rightrightarrows\Delta_\rho,\quad \tilde\Phi_\alpha(x)=B_\alpha(\Phi(x))\cap\Delta_\rho,\quad
\tilde{\mc{R}}_h(t_0,X_0)=B_{\alpha}(X_0)\cap\Delta_\rho&\\ 
&\tilde{\mc{R}}_h(t_{n+1},X_0) = \tilde{\Phi}_{\alpha}(t_n,\tilde{\mc{R}}_h(t_n,X_0)),\ n=0,\ldots,N-1.&
\end{eqnarray*}
for some $\alpha\ge\rho/2$, because this setting guarantees that $\tilde\Phi_\alpha(\cdot)$ is a good approximation of $\Phi(\cdot)$.
Spatial discretization and its impact on error estimates have been discussed in \cite{Beyn:Rieger:07}.

\medskip

The computation of reachable sets by the fully discrete Euler scheme is very costly, which is mainly due to the fact that
the images $\tilde\Phi_\alpha(x)$ and $\tilde\Phi_\alpha(\tilde x)$ overlap if $x$ and $\tilde x$ are close to each other
and the computed information is highly redundant. Several attempts have been made to develop faster numerical methods.

In \cite{Lempio:Veliov:98} and \cite{Veliov:89}, a multivalued version of Heun's method and more general Runge-Kutta methods
are shown to converge quadratically w.r.t.\ the time-step under strong assumptions on $F$ when considered without spatial discretization.
Successive evaluations of the discretized multivalued right-hand sides, however, are so costly that in practice, the performance
of such methods is worse than that of the Euler scheme. 

The paper \cite{Sandberg:08} investigates the properties of the Euler scheme for differential inclusions with non-convex right-hand sides
and, in particular, proves first-order convergence in this situation.
As the closures of the reachable sets of a differential inclusion with non-convex right-hand side coincide
with the reachable sets of the convexified problem, this result demonstrates that it is possible to approximate the desired sets
using only extremal points of the right-hand sides for computing the Euler step without losing the order of convergence. Nevertheless,
the resulting version of the Euler scheme has a larger error than its classical pendant, so that the gain of efficiency is moderate.

An unconventional approach to the problem has recently been published in \cite{Baier:Gerdts:Xausa:13}, where the defining 
relations of the Euler or a more general Runge-Kutta scheme are considered as constraints in an optimization routine that
aims to minimize the distance between given points in phase space and the states that can be reached by trajectories of the
numerical scheme. 
This method reduces unnecessary computations efficiently, but bears the risk of losing parts of the reachable set by failure
of the optimizer to find the global minimum.

\medskip

The boundary Euler method proposed and analyzed in the present paper uses a simple but effective strategy to reduce computational costs.
It tracks the boundaries (precise definition in Section 3) of the reachable sets 
of the fully discretized Euler scheme instead of computing the whole reachable sets 
and computes a second layer of 
exterior points close to the boundary  in every time-step from a layer of exterior points in the preimage. 
The boundary and the exterior layer together contain all necessary information about the topology of the discrete reachable sets.
Moreover, the algorithm can be arranged in such a way that instead of full images $F(\cdot)$ 
it is essentially sufficient to work with the boundary $\partial F(\cdot)$ of the right-hand side.
As the boundary of the fully discrete Euler scheme is exactly reproduced, the boundary Euler scheme 
satisfies the same error estimates.

It is surprisingly hard to prove that the boundary Euler method indeed computes the boundaries 
of the discrete reachable sets. To this end, Section 2 gathers some analytical prerequisites  
and Section 3 exploits them to set up a preliminary version of the method. 
The drawback of this preliminary version is that full images $F(\cdot)$
of the right-hand side have to be computed in order to ensure the propriety of the algorithm. Nontrivial
topological arguments given in Section 4 improve the results of Section 2 and facilitate the formulation of the boundary Euler method
in its final form. 

\medskip

The boundary Euler scheme is at present the fastest numerical method for non-stiff ordinary differential inclusions.
Up to our knowledge, the result as well as some of the technical tools necessary for the proof have never been considered in the literature.
Throughout the paper, great care is taken to obtain optimal constants, because
suboptimal constants strongly increase the computational costs of the boundary Euler method without increasing its precision.

\section{Notation and analytical prerequisites}

The notation used in this paper is mostly standard. For any set $A\subset\R^d$, the symbols $A^c$, $\interior A$, and $\partial A$
denote the complement, the interior, and the boundary of $A$. For $p\in[1,\infty]$, the one-sided and the symmetric Hausdorff distances 
(induced by the $p$-norm) between compact sets $A,B\subset\R^d$ are given by 
$\dist(A,B)_p:=\sup_{a\in A}\inf_{b\in B}|a-b|_p$ 
and $\dist_H(A,B)_p:=\max\{\dist(A,B)_p,\dist(B,A)_p\}$. For any $x\in\R^d$, we set $\Proj(x,A)_p:=\{a\in A: |x-a|_p=\dist(x,A)_p\}$,
and for $r\ge 0$ we define $B_r(A)_p:=\{x\in\R^d: \dist(x,A)_p\le r\}$. 

Let $F:\R^{d_1}\rightrightarrows\R^{d_2}$ be a set-valued mapping with compact images. Then $F(\cdot)$ is called $L$-Lipschitz
if $\dist_H(F(x),F(\tilde x))_p \le L|x-\tilde x|_p$ for all $x,\tilde x\in\R^{d_1}$. 
Let $\CoCpSets(\R^d)$ denote the collection of all subsets of $\R^d$ that are convex and compact.
It is well-known that a mapping $F:\R^{d_1}\rightarrow\CoCpSets(\R^{d_2})$ is Lipschitz if and only if its boundary $\partial F(\cdot)$ 
is Lipschitz with the same Lipschitz constant (see \cite{Monteiro:Marques:84}).

The identity mapping is denoted $id:\R^d\rightarrow\R^d$. For any two paths $\phi,\tilde\phi\in C([0,1],\R^d)$ with $\phi(1)=\tilde\phi(0)$,
their concatenation $\tilde\phi\circ\phi\in C([0,1],\R^d)$ is defined by 
\[(\tilde\phi\circ\phi)(\lambda):=\left\{ \begin{array}{ll} \phi(2\cdot\lambda),&0\le\lambda\le 1/2,\\ 
\tilde\phi(2\cdot\lambda-1),&1/2\le\lambda\le 1.\end{array}\right.\]

\medskip

The following result is related to \cite[Theorem 2]{Beyn:Rieger:10}, where more general mappings are discussed
in the Euclidean $\R^d$. 
For the purposes of the present paper, however, it will be essential to measure distances in $|\cdot|_\infty$.
An error term of size $\sqrt{d}$ caused by embedding $(\R^d,|\cdot|_2)$ into $(\R^d,|\cdot|_\infty)$ is not tolerable,
because suboptimal estimates force the algorithm to carry out unnecessary computations.

\begin{proposition} \label{solvable}
Let $G:\R^d\rightarrow\CoCpSets(\R^d)$ be $l$-Lipschitz w.r.t.\ $|\cdot|_\infty$ with $l<1$. 
Then for any  $x^0,\hat y\in\R^d$ there exists $\hat x\in\R^d$ such that $\hat y\in \hat x + G(\hat x)$ and 
\begin{equation} \label{tag:1}
|x^0-\hat x|_\infty \le \frac{\dist(\hat y,x^0+G(x^0))_\infty}{1-l}.
\end{equation}
If $\hat y\notin x^0+G(x^0)$, then there exists $\hat x$ satisfying \eqref{tag:1} and $\hat y\in \hat x+\partial G(\hat x)$.
\end{proposition}
\begin{proof}
Construct successively
\begin{align*}
g^k \in \Proj(\hat y-x^k,G(x^k))_\infty,\quad r^k := \hat y-(x^k+g^k),\quad
x^{k+1} := x^k+r^k
\end{align*}
for $k\in\N$. Then 
\begin{align*}
|r^{k+1}|_\infty &= \dist(\hat y,x^{k+1}+G(x^{k+1}))_\infty = \dist(\hat y, x^k+r^k+G(x^{k+1}))_\infty\\ 
&= \dist(g^k,G(x^{k+1}))_\infty \le \dist(G(x^k),G(x^{k+1}))_\infty\\
&\le l|x^{k+1}-x^k|_\infty \le l|r^k|_\infty
\end{align*}
for $k\in\N$. Hence
\[|x^n-x^m|_\infty \le \sum_{k=m}^{n-1}|r^k|_\infty\le |r^0|_\infty\sum_{k=m}^{n-1}l^k \le\frac{l^m}{1-l}|r^0|_\infty\]
for all $m,n\in\N$ with $m\le n$, so that $\{x^k\}_{k=0}^\infty$ is Cauchy and $\hat x := \lim_{k\rightarrow\infty}x^k$ exists. 
Then
\[|\hat x-x^0|_\infty \le \sum_{k=0}^\infty|r^k|_\infty \le \frac{|r^0|_\infty}{1-l}
=\frac{\dist(\hat y,x^0+G(x^0))_\infty}{1-l},\]
and
\[\dist(\hat y,\hat x+G(\hat x))_\infty = \lim_{k\rightarrow\infty}\dist(\hat y,x^k+G(x^k))_\infty 
= \lim_{k\rightarrow\infty}|r^k|_\infty = 0\]
together with $G(\hat x)\in\CoCpSets(\R^d)$ imply $\hat y\in\hat x+G(\hat x)$.

\medskip

Consider the case $\hat y\notin x^0+G(x^0)$. Construct $\hat x\in\R^d$ with $\hat y\in \hat x+G(\hat x)$ 
and \eqref{tag:1} as above. Define $\phi(\lambda):=\lambda\hat x + (1-\lambda)x^0$ for $\lambda\in[0,1]$.
Let 
\[\lambda^* := \inf\{\lambda\in[0,1]: \hat y\in\phi(\lambda)+G(\phi(\lambda))\}.\]
There exists a sequence $(\lambda_n)_{n\in\N}\subset[0,1]$ with $\lambda_n\searrow\lambda^*$ and
$\hat y\in\phi(\lambda_n)+G(\phi(\lambda_n))$, so that by continuity $\hat y\in\phi(\lambda^*)+G(\phi(\lambda^*))$.

Assume that $\hat y\in\phi(\lambda^*)+\interior G(\phi(\lambda^*))$. Then the properties of $G$ ensure that
there exists $\epsilon>0$ such that 
\begin{equation} \label{tag:2}
\hat y\in\phi(\lambda)+G(\phi(\lambda))\quad\text{for all}\ \lambda\in(\lambda^*-\epsilon,\lambda^*+\epsilon),
\end{equation}
which contradicts minimality of $\lambda^*$. Hence $\hat y\in\phi(\lambda^*)+\partial G(\phi(\lambda^*))$,
and by construction $\phi(\lambda^*)$ satisfies \eqref{tag:1}.
\end{proof}

Obviously, the (time-independent) Euler map $\Phi(x)=x+hF(x)$ satisfies the assumptions of Proposition \ref{solvable}
if $F:\R^d\rightarrow\CoCpSets(\R^d)$ is $L$-Lipschitz and $Lh<1$.

\begin{remark}
The fact that $G$ has convex images is only used in \eqref{tag:2}. One could think of weaker assumptions such as
$l$-Lipschitz continuity of a compact-valued mapping $G:\R^d\rightrightarrows\R^d$ and continuity of the complement 
$G^c:\R^d\rightrightarrows\R^d$ to prevent the sudden formation of holes. This is, however, not the focus of this paper.
\end{remark}

In the following proposition, it is again important to work in the maximum norm to avoid embedding constants 
that may lead to unnecessary restrictions of the step size $h$.

\begin{proposition} \label{connected}
Let $G:\R^d\rightarrow\CoCpSets(\R^d)$ be $l$-Lipschitz w.r.t.\ $|\cdot|_\infty$ with $l<1$.
Then for any $\xi\in\R^d$ the set $(\id+G)^{-1}(\xi):=\{x\in\R^d: \xi\in x+G(x)\}$
is nonempty and path-connected.
\end{proposition}
\begin{proof}
By Proposition \ref{solvable} we have $(\id+G)^{-1}(\xi)\neq\emptyset$. 
Let $z,\tilde z\in(\id+G)^{-1}(\xi)$. Define $x^0:[0,1]\rightarrow\R^d$ by
\[x^0(\lambda):=\lambda z+(1-\lambda)\tilde z.\]
For $k\in\N$ set $p(k):=k^2+2$ and define successively functions $f^k,r^k,x^{k+1}:[0,1]\rightarrow\R^d$ by 
\begin{align*}
g^k(\lambda) &:= \Proj(\xi-x^k(\lambda),G(x^k(\lambda))))_{p(k)},\\
r^k(\lambda) &:= \xi-(x^k(\lambda)+g^k(\lambda)),\\
x^{k+1}(\lambda) &:= x^k(\lambda)+r^k(\lambda),
\end{align*}
where $\Proj(\cdot,\cdot)_{p}$ denotes the projection w.r.t.\ $|\cdot|_{p}$.
These functions are continuous, because $(\R^d,|\cdot|_p)$ is strictly convex for $1<p<\infty$ 
(see \cite[Section 9.3]{Aubin:Frankowska:90}).
Moreover, $x^k(0)=z$ and $x^k(1)=\tilde z$ hold for all $k\in\N$,
and $g^k(\lambda)\in G(x^k(\lambda))$ for all $\lambda\in[0,1]$ and $k\in\N$.

\medskip

Since $\lambda\mapsto\dist(\xi,x^0(\lambda)+G(x^0(\lambda)))_\infty$ is a continuous function from the compact interval
$[0,1]$ to the real numbers, $\|r^0\|_\infty<\infty$. In addition,
\begin{align*}
&|r^{k+1}(\lambda)|_\infty  \le |r^{k+1}(\lambda)|_{p(k+1)} 
= \dist(\xi,x^{k+1}(\lambda)+G(x^{k+1}(\lambda)))_{p(k+1)}\\
&= \dist(g^k(\lambda),G(x^{k+1}(\lambda)))_{p(k+1)} 
\le \dist(G(x^k(\lambda)),G(x^{k+1}(\lambda)))_{p(k+1)}\\
&\le d^{\frac{1}{p(k+1)}}\dist(G(x^k(\lambda)),G(x^{k+1}(\lambda)))_\infty
\le d^{\frac{1}{p(k+1)}}l|r^k(\lambda)|_\infty,
\end{align*}
so that
\begin{align*}
\|r^k\|_\infty &\le (\prod_{j=1}^kd^\frac{1}{p(j)})l^k\|r^0\|_\infty 
\le (\prod_{j=1}^kd^\frac{1}{j^2})l^k\|r^0\|_\infty\\
&= d^{\sum_{j=1}^k\frac{1}{j^2}}l^k\|r^0\|_\infty \le d^\frac{\pi^2}{6}l^k\|r^0\|_\infty.
\end{align*}
For $m,n\in\N$ with $m\le n$, it follows that 
\begin{align*}
\|x^n-x^m\|_\infty &\le \sum_{k=m}^{n-1}\|r^k\|_\infty 
\le d^\frac{\pi^2}{6}\|r^0\|_\infty \sum_{k=m}^{n-1}l^k 
\le d^\frac{\pi^2}{6}\|r^0\|_\infty\frac{l^m}{1-l}.
\end{align*}
Hence $\{x^k\}_{k=0}^\infty\subset (C([0,1],\R^d),\|\cdot\|_\infty)$ is Cauchy,
and there exists $\hat x\in C([0,1],\R^d)$ with $\lim_{k\rightarrow\infty}\|x^k-\hat x\|_\infty = 0$.
It is clear that $\hat x(0)=z$ and $\hat x(1)=\tilde z$. Finally,
\begin{align*}
0 \leftarrow d^\frac{1}{p(k)}\|r^k\|_\infty &\ge \|r^k\|_{p(k)} = \dist(\xi,(\id+G)(x^k(\lambda)))_{p(k)} \\
&\ge \dist(\xi,(\id+G)(x^k(\lambda)))_\infty \rightarrow \dist(\xi,(\id+G)(\hat x(\lambda)))_\infty
\end{align*}
for $k\rightarrow\infty$ implies $\xi\in(\id+G)(\hat x(\lambda))$ for all $\lambda\in[0,1]$.
\end{proof}

The following observation is the main motivation for the development of the boundary Euler method.

\begin{proposition} \label{boundary:from:boundary}
Let $F:\R^d\rightarrow\CoCpSets(\R^d)$ be $L$-Lipschitz, let $h>0$ be so small that $Lh<1$, and let $M\subset\R^d$ be compact.
\begin{itemize}
\item [a)] If $x\in M$ and $\delta:=\dist(x,\partial M)_\infty>0$, then 
\[B_{(1-Lh)\delta}(\Phi(x))_\infty \subset \Phi(M)\]
\item [b)] For every $y\in\partial(\Phi(M))$, we have $\emptyset\neq\Phi^{-1}(y)\cap M\subset\partial M$,
and for every $x\in\partial M$ with $y\in\Phi(x)$ we have $y\in\partial\Phi(x)$.
\end{itemize}
\end{proposition}
\begin{proof}
\begin{itemize}
\item [a)] Let $y\in\Phi(x)$ and $\tilde y\in B_{(1-Lh)\delta}(y)_\infty$. By Proposition \ref{solvable},
there exists $\tilde x \in\R^d$ such that $\tilde y\in\Phi(\tilde x)$ and
\[|\tilde x-x|_\infty \le \frac{\dist(\tilde y,\Phi(x))_\infty}{1-Lh} \le \frac{|\tilde y-y|_\infty}{1-Lh} \le \delta,\]
so that $\tilde x\in M$ and $\tilde y\in\Phi(M)$.
\item [b)] For every $y\in\partial(\Phi(M))$, there exists $x\in M$ with $y\in\Phi(x)$.
By a), we have $x\in\partial M$. Since $y\in\interior\Phi(x)$ contradicts $y\in\partial(\Phi(M))$, it follows that $y\in\partial\Phi(x)$.
\end{itemize}
\end{proof}

Proposition \ref{boundary:from:boundary} tells us that $\partial\Phi(M)$ can be entirely reconstructed from the
boundaries of images of $\partial M$, but it does not provide any clue how such a reconstruction could be achieved. 
Counterexamples not included here indicate that there is no simple algorithm for this task.

\section{The fully discrete boundary Euler}

In the following, we set up the terminology for handling discrete sets and the space-discrete Euler map.
Fix a grid $\Delta_\rho:=\rho\Z^d\subset\R^d$ and let $A\subset\R^d$. Then
\begin{align*}
\partial_\rho^0 A &:= \{a\in A\cap\Delta_\rho: \exists x\in\Delta_\rho\setminus A\ \text{s.t.}\ |x-a|_\infty=\rho\},\\
\interior_\rho A &:= (A\cap\Delta_\rho)\setminus\partial_\rho^0 A,\\
\partial_\rho^k A &:= \{x\in\Delta_\rho\setminus A: \dist(x,\partial^0 A)=k\rho\},\ k\in\N_1,\\
\partial_\rho^{-k} A &:= \{a\in A\cap\Delta_\rho: \dist(a,\partial^0 A)=k\rho\},\ k\in\N_1
\end{align*}
are the discrete equivalents of the boundary and the interior of $A$, layers in the complement and layers in the
interior of $A$. 

As usual, it is necessary to take blowups 
\begin{align*}
\Phi_\alpha(x) := B_\alpha(\Phi(x))\quad\text{and}\quad \Phi_\alpha^\partial(x) := B_\alpha(\partial\Phi(x))
\end{align*}
with $\alpha\ge\rho/2$ of all maps under consideration, so that their intersection with $\Delta_\rho$ is
a well-defined $\alpha$-close approximation of the original mappings w.r.t.\ the Hausdorff distance (see e.g.\ \cite{Beyn:Rieger:07}). 
Moreover, it may happen that complicated values $F(x)$, $x\in\R^d$,
of the right-hand side cannot be computed exactly. In that case, it is possible to use overapproximations 
by convex polytopes as analyzed in \cite[Lemma 19]{Rieger:11}. The impact on the error of the Euler scheme
is estimated in \cite[Proposition 4]{Beyn:Rieger:07}. We will therefore consider maps $\Phi_{\alpha,\beta}(\cdot)$
and $\Phi_{\alpha,\beta}^\partial(\cdot)$ with
\begin{align*}
\Phi_{\alpha}(\cdot)\subset\Phi_{\alpha,\beta}(\cdot)\subset B_\beta(\Phi_{\alpha}(\cdot))\quad\text{and}\quad
\Phi_{\alpha}^\partial(\cdot)\subset\Phi_{\alpha,\beta}^\partial(\cdot)\subset B_\beta(\Phi_{\alpha}^\partial(\cdot)) 
\end{align*}
for all $x\in\R^d$. The results below will show that it makes sense to use step-sizes $0<h\le h^* := \frac{1}{4L}$ and fixed parameters
\[\alpha^*:=(1+Lh)\rho/2,\quad 0\le\beta^*<\min\{(1-3Lh)\rho,(1-Lh)\rho/2\}.\]
In the following, $M\subset\Delta_\rho$ will be a compact set. We will show that the boundary Euler is well-defined on $M$
in the sense that it correctly computes $\partial_\rho^0(\Phi_{\alpha^*,\beta^*}(M))$ and $\partial_\rho^1(\Phi_{\alpha^*,\beta^*}(M))$
from $\partial_\rho^0M$ and $\partial_\rho^1M$. Therefore, the boundary Euler method computes the same reachable sets
as the original Euler scheme $\Phi_{\alpha^*,\beta^*}(\cdot)\cap\Delta_\rho$ discussed in \cite{Beyn:Rieger:07}.

\medskip

Proposition \ref{boundary} shows that knowing $\partial_\rho^0M$ and $\partial_\rho^{-1}M$ suffices to compute 
a superset of $\partial_\rho^0(\Phi_{\alpha^*,\beta^*}(M))$. We do not assume that the discrete interior is nonempty
or has any particular properties.
\begin{proposition} \label{boundary}
For every $y\in\partial_\rho^0(\Phi_{\alpha^*,\beta^*}(M))$ there exists $x\in\partial_\rho^0M\cup\partial_\rho^{-1}M$ 
with $y\in\Phi_{\alpha^*}^\partial(x)$.
\end{proposition} 
\begin{proof}
For $y\in\partial_\rho^0(\Phi_{\alpha^*,\beta^*}(M))$ there exist $x\in M$ with $y\in\Phi_{\alpha^*,\beta^*}(x)$ and
$\eta\in\partial_\rho^1(\Phi_{\alpha^*,\beta^*}(M))$ such that $|\eta-y|_\infty=\rho$. Then
\[\dist(\eta,\Phi(x))_\infty = |\eta-y|_\infty + \dist(y,\Phi(x))_\infty \le \rho+\alpha^*+\beta^*,\]
so that by Proposition \ref{solvable} there exists $z\in\R^d$ with $\eta\in\Phi(z)$ and
\[|x-z|_\infty \le \frac{\rho+\alpha^*+\beta^*}{1-Lh}.\]
There exists $\tilde z\in B_{\rho/2}(z)\cap\Delta_\rho$. We have
\begin{align*}
\dist(\eta,\Phi_{\alpha^*}(\tilde z))_\infty \le \dist(\eta,\Phi(z))_\infty + \dist(\Phi(z),\Phi_{\alpha^*}(\tilde z))_\infty = 0,
\end{align*}
so that $\eta\in\Phi_{\alpha^*,\beta^*}(\tilde z)$ and hence $\tilde z\in M^c\cap\Delta_\rho$. Because of
\begin{align*}
|x-\tilde z|_\infty \le |x-z|_\infty + |z-\tilde z|_\infty \le \frac{\rho+\alpha^*+\beta^*}{1-Lh}+\frac{\rho}{2} < 3\rho,
\end{align*}
we have $x\in\partial_\rho^0M\cup\partial_\rho^{-1}M$.

Assume that $y\notin\Phi_{\alpha^*}^\partial(x)$. Then for all $\tilde y\in B_\rho(y)_\infty\cap\Delta_\rho$, we have
\[\tilde y \in B_\rho(y)_\infty \subset B_{2\alpha^*}(y)_\infty \subset \Phi_{\alpha^*}(x),\]
which contradicts $y\notin\partial_\rho^0(\Phi_{\alpha^*,\beta^*}(M))$.
\end{proof}

The following proposition tells us that a superset of the outer layer $\partial_\rho^1(\Phi_{\alpha^*,\beta^*}(M))$ can be 
computed from two outer layers of the preimage. 
\begin{proposition} \label{layer}
For every $\eta\in\partial_\rho^1(\Phi_{\alpha^*,\beta^*}(M))$, there exists $z\in\partial_\rho^1M\cup\partial_\rho^2M$ such that
$\eta\in\Phi_{\alpha^*,\beta^*}^\partial(z)$.
\end{proposition}
\begin{proof}
There exist $\xi\in\partial_\rho^0(\Phi_{\alpha^*,\beta^*}(M))$ with $|\xi-\eta|_\infty = \rho$ and
some $x\in M$ such that $\xi\in\Phi_{\alpha^*,\beta^*}(x)$. Since $\eta\notin\Phi_{\alpha^*,\beta^*}(M)$, we have 
$\eta\in \Phi_{\alpha^*,\beta^*}(x)^c\subset \Phi(x)^c$. Now
\[\dist(\eta,\Phi(x))_\infty \le |\eta-\xi|_\infty + \dist(\xi,\Phi(x))_\infty \le \rho+\alpha^*+\beta^*,\]
and by Proposition \ref{solvable} there exists $\tilde x\in\R^d$ with $\eta\in\partial\Phi(\tilde x)$ and
\[|x-\tilde x|_\infty \le \frac{\dist(\eta,\Phi(x))_\infty}{1-Lh} \le \frac{\rho+\alpha^*+\beta^*}{1-Lh}.\]
There exists $z\in B_{\rho/2}(\tilde x)\cap\Delta_\rho$, and since $\partial\Phi(\cdot)$ is $Lh$-Lipschitz, 
we have
\begin{align*}
\dist(\eta,\Phi_{\alpha^*}^\partial(z))_\infty \le \dist(\eta,\partial\Phi(\tilde x)) 
+ \dist(\partial\Phi(\tilde x),\Phi_{\alpha^*}^\partial(z))_\infty = 0,
\end{align*}
so that $\eta\in\Phi_{\alpha^*,\beta^*}^\partial(z)$. Moreover,
\[|x-z|_\infty \le |x-\tilde x|_\infty + |\tilde x-z|_\infty \le \frac{\rho+\alpha^*+\beta^*}{1-Lh} + \frac{\rho}{2} < 3\rho,\]
and hence $\dist(z,M)_\infty\le 2\rho$.
\end{proof}

In the propositions above, we guaranteed that we can compute supersets of $\partial_\rho^0(\Phi_{\alpha^*,\beta^*}(M))$
and $\partial_\rho^1(\Phi_{\alpha^*,\beta^*}(M))$.
We now have to ensure that we can get rid of the unwanted parts of these supersets.

\begin{proposition} \label{intersection}
Let $y\in M^c\cap\Delta_\rho$, $z\in\interior_\rho M$, and $\xi\in\R^d$ be points such that
$\xi\in\Phi_{\alpha^*,\beta^*}(y)\cap\Phi_{\alpha^*,\beta^*}(z)$. Then there exists $x\in\partial_\rho^0 M$ with $\xi\in\Phi_{\alpha^*}(x)$.
\end{proposition}
\begin{proof}
By Proposition \ref{solvable} there exists $\tilde y\in\Phi^{-1}(\xi)$ satisfying  
\begin{align*}|\tilde y-y|_\infty \le \frac{\dist(\xi,\Phi(y))_\infty}{1-Lh} \le \frac{\alpha^*+\beta^*}{1-Lh} = \rho.
\end{align*}
For the same reason, there exists $\tilde z\in\Phi^{-1}(\xi)$ with $|\tilde z-z|_\infty\le\rho$.
By Proposition \ref{connected} there exists a continuous $\phi:[0,1]\rightarrow\Phi^{-1}(\xi)$
with $\phi(0)=\tilde y$ and $\phi(1)=\tilde z$. Since $\dist(\phi(0),\interior_\rho M)_\infty\ge\rho$
and $\dist(\phi(1),\interior_\rho M)_\infty\le\rho$, continuity implies the existence of $\lambda^*\in[0,1]$
satisfying $\dist(\phi(\lambda^*),\interior_\rho M)_\infty=\rho$. There exists $x\in\Delta_\rho$ with 
$|\phi(\lambda^*)-x|_\infty\le\frac{\rho}{2}$. Now $\frac12\rho\le\dist(x,\interior_\rho M)_\infty\le\frac32\rho$,
implies $x\in\partial^0M$, and
\[\dist(\xi,\Phi_{\alpha^*}(x))_\infty \le \dist(\Phi(\phi(\lambda^*)),\Phi_{\alpha^*}(x))_\infty =0,\]
so that $\xi\in\Phi_{\alpha^*}(x)$.
\end{proof}

Propositions \ref{boundary}, \ref{layer}, and \ref{intersection} enable the following preliminary boundary Euler algorithm.
A justification is given in Theorem \ref{summary}.

\bigskip

\noindent\fbox{\begin{minipage}{\textwidth}
Assume that the discrete sets $\partial_\rho^0M$ and $\partial_\rho^1M$ are known.
 \begin{itemize}
\item [1.] Compute $\partial_\rho^{-1}M$ and $\partial_\rho^2M$.
\item [2.] Compute the discrete sets 
$S_0:=(\Phi_{\alpha^*,\beta^*}(\partial_\rho^0M)\cup\Phi_{\alpha^*,\beta^*}^\partial(\partial_\rho^{-1}M))\cap\Delta_\rho$ \\
and $S_1:=(\Phi_{\alpha^*,\beta^*}^\partial(\partial_\rho^1M\cup\partial_\rho^2M))\cap\Delta_\rho$.
\item [3.] Compute $\partial_\rho^1(\Phi_{\alpha^*,\beta^*}(M))=\{x\in S_1: \dist(x,S_0)=\rho\}$.
\item [4.] Compute $\partial_\rho^0(\Phi_{\alpha^*,\beta^*}(M))=\{x\in S_0: \dist(x,\partial_\rho^1(\Phi_{\alpha^*,\beta^*}(M)))=\rho\}$.
\end{itemize}
\end{minipage}}

\bigskip

Step 1 is trivial, and step 2 is just the application of the fully discrete Euler scheme.
Steps 3 and 4 can be realized in one search process. As all sets are sparse, they should be stored  
in binary trees rather than arrays of booleans (see \cite{Beyn:Rieger:07}).

\begin{theorem} \label{summary}
The preliminary boundary Euler scheme is well-defined.
\end{theorem}
\begin{proof}
By Proposition \ref{boundary} we have 
\begin{equation} \label{tag:3}
\partial_\rho^0(\Phi_{\alpha^*,\beta^*}(M)) \subset S_0 \subset \Phi_{\alpha^*,\beta^*}(M),
\end{equation}
and by Proposition \ref{layer} we have $\partial_\rho^1(\Phi_{\alpha^*,\beta^*}(M)) \subset S_1$,
so that
\[\partial_\rho^1(\Phi_{\alpha^*,\beta^*}(M))\subset\{x\in S_1: \dist(x,S_0)=\rho\}.\] 
Since Proposition \ref{intersection} ensures $S_1\setminus S_0\subset\Phi_{\alpha^*,\beta^*}(M)^c$,
it follows that
\[\partial_\rho^1(\Phi_{\alpha^*,\beta^*}(M))\supset\{x\in S_1: \dist(x,S_0)=\rho\}.\] 
The equality 
\[\partial_\rho^0(\Phi_{\alpha^*,\beta^*}(M))=\{x\in S_0: \dist(x,\partial_\rho^1(\Phi_{\alpha^*,\beta^*}(M)))=\rho\}\]
follows from \eqref{tag:3}.
\end{proof}

In step 2, the full images $\Phi_{\alpha^*,\beta^*}(\partial_\rho^0M)$ have to be computed.
Simple examples indicate that it should be sufficient to work with a version of $\Phi_{\alpha^*}^\partial(\cdot)$ 
that is blown up slightly into the interior of $\Phi_{\alpha^*}(\cdot)$. 
This guess is stated more precisely and analyzed thoroughly in the next section.

\section{Topological considerations}

This whole section is concerned with an improvement of Proposition \ref{intersection}
that requires the computation of full images of $\Phi_{\alpha^*}(\cdot)$. The overall idea is the following:
If $y\in M^c\cap\Delta_\rho$ and $z\in\interior_\rho M$ are points with $\xi\in\Phi_{\alpha^*,\beta^*}^\partial(y)$ 
and $\xi\in\Phi_{\alpha^*,\beta^*}(z)$,
then it is easy to construct two points $v_y,v_z\in\partial_\rho^0M$ such that 
$\dist(\xi,\Phi_{\alpha^*}^\delta(v_y))\le\kappa$ for small $\kappa>0$ and $\xi\in\Phi_{\alpha^*}(v_z)$. 
It is, however, nontrivial to construct a point $v\in\partial_\rho^0M$
satisfying both conditions at once. To this end, the global Leray-Schauder theorem is used to join $v_y$ and $v_z$ 
by a compact connected set $C$ in a suitable way, and we prove the existence of a $v\in C\cap\partial_\rho^0M$ with the desired properties. 

At first, we introduce topological notions on grid sets and prove that the images of the Euler scheme are connected in a certain sense.
Then we prepare the original problem for the application of the Leray-Schauder theorem and draw the necessary conclusions.

\begin{definition}
Let $x,\tilde x\in\Delta_\rho$
A sequence $c=(c_n)_{n=0}^N\subset\Delta_\rho$ with $N\in\N$ is called a chain connecting $x$ and $\tilde x$ 
if $N=0$ and $c_0=x=\tilde x$ or $N>0$, $c_0=x$, $c_N=\tilde x$, and
\[|c_{n+1}-c_n|_\infty = \rho,\quad n=0,\ldots,N-1.\]
The collection of all chains connecting $x$ and $\tilde x$ will be called $\chain(x,\tilde x)$.

Let $x,\tilde x,\hat x\in\Delta_\rho$, $c=(c_n)_{n=0}^N\in\chain(x,\tilde x)$, and $\tilde c=(\tilde c_n)_{n=0}^{\tilde N}\in\chain(\tilde x,\hat x)$.
Then the concatenation of $c$ and $\tilde c$ is defined by
\[\tilde c\circ c:=(x,c_1,\ldots,c_{N-1},\tilde x,\tilde c_1,\ldots,\tilde c_{\tilde N-1},\hat x)\in\chain(x,\hat x).\]

A set $M\subset\Delta_\rho$ is called chain-connected if for any two $x,\tilde x\in M$ there exists a chain $c\in\chain(x,\tilde x)$
such that $c\subset M$.
\end{definition}

\subsection{Chain-connectedness of reachable sets}

The following results show that $\Phi_{\alpha^*}(M)\cap\Delta_\rho$ is chain-connected whenever $M$ has this property.
This means that all reachable sets of the Euler scheme are chain-connected if we require the initial
set $X_0$ to be chain-connected. This will henceforth be assumed.

\begin{lemma} \label{existence:of:grid:point}
Let $x\in\R^d$ and $z\in B_{\rho/2}(x)_\infty$. Then there exists $\eta\in B_{\rho/2}(x)\cap\Delta_\rho$
such that $|z-\eta|_\infty<\rho$.
\end{lemma}
\begin{proof}
There exists $\xi\in\Delta_\rho$ such that $|\xi-z|_\infty\le\rho/2$. The point $\eta\in\R^d$ defined by
\[\eta_n:=\left\{\begin{array}{ll}\xi_n, & |\xi_n-x_n|\le\rho/2\\ \xi_n-\text{sign}(\xi_n-x_n)\rho,&|\xi_n-x_n|>\rho/2\end{array}\right.\]
has the desired properties: If $|\xi_n-x_n|\le\rho/2$, then $|\eta_n-z_n|=|\xi_n-z_n|\le\rho/2$.
If $\xi_n-x_n > \rho/2$, then $\xi_n>z_n>x_n$, and $|z_n-\eta_n| = z_n-(\xi_n-\rho) < \rho$.
The case $x_n-\xi_n > \rho/2$ is symmetric.
\end{proof}

The fact that individual images $\Phi_{\alpha^*}(x)\cap\Delta_\rho$ of points $x\in M$ are chain connected
will enable us to show in Proposition \ref{image:chain:connected} that their union is connected as well.
\begin{lemma} \label{single:image:connected}
The set $\Phi_{\alpha^*}(x)\cap\Delta_\rho$ is nonempty and chain-connected for any $x\in\R^d$.
\end{lemma}
\begin{proof}
Since $\alpha^*\ge\rho/2$, it follows that $\Phi_{\alpha^*}(x)\cap\Delta_\rho\neq\emptyset$.

Assume that there exist $y,\tilde y\in\Phi_{\alpha^*}(x)\cap\Delta_\rho$ such that 
$\{c\in\chain(y,\tilde y): c\subset\Phi_{\alpha^*}(x)\}=\emptyset$. 
Consider the sets 
$K:=\{z\in\Phi_{\alpha^*}(x)\cap\Delta_\rho: \exists c\in\chain(y,z), c\subset\Phi_{\alpha^*}(x)\}$,
$\tilde K:=\{z\in\Phi_{\alpha^*}(x)\cap\Delta_\rho: \exists c\in\chain(\tilde y,z), c\subset\Phi_{\alpha^*}(x)\}$,
and $\hat K:=(\Phi_{\alpha^*}(x)\cap\Delta_\rho)\setminus(K\cup\tilde K)$. 

By construction, the sets $\interior B_{\rho}(K)_\infty$, $\interior B_{\rho}(\tilde K)_\infty$, 
and $\interior B_{\rho}(\hat K)_\infty$ are pairwise disjoint: 
If there exists $\eta\in\interior B_{\rho}(K)_\infty\cap\interior B_{\rho}(\tilde K)_\infty$,
then there exist points $\xi\in K$ and $\tilde\xi\in\tilde K$ such that $|\xi-\tilde\xi|_\infty<2\rho$.
As $\xi,\tilde\xi\in\Delta_\rho$, $|\xi-\tilde\xi|_\infty\le\rho$ ensues. By definition, there exist
$c\in\chain(y,\xi)$ and $\tilde c\in\chain(\tilde\xi,\tilde y)$ with $c,\tilde c\subset\Phi_{\alpha^*}(x)$.
But then $\tilde c\circ(\xi,\tilde\xi)\circ c\subset\Phi_{\alpha^*}(x)$ connects $y$ and $\tilde y$, which is a contradiction.
A similar argument shows that $\interior B_{\rho}(K)_\infty\cap\interior B_{\rho}(\hat K)_\infty=\emptyset$
and $\interior B_{\rho}(\tilde K)_\infty\cap\interior B_{\rho}(\hat K)_\infty=\emptyset$.

Let $z\in\Phi_{\alpha^*}(x)$. Then there exists some $v\in\R^d$ such that $z\in B_{\rho/2}(v)_\infty\subset\Phi_{\alpha^*}(x)$.
By Lemma \ref{existence:of:grid:point}, there exists $\eta\in B_{\rho/2}(v)_\infty\cap\Delta_\rho\subset\Phi_{\alpha^*}(x)\cap\Delta_\rho$
such that $|z-\eta|_\infty<\rho$. Therefore,
$\Phi_{\alpha^*}(x) \subset \interior B_{\rho}(K)_\infty \cup \interior B_{\rho}(\tilde K)_\infty \cup \interior B_{\rho}(\hat K)_\infty$, so that $\Phi_{\alpha^*}(x)$ is not connected, which contradicts convexity of $\Phi_{\alpha^*}(x)$.
\end{proof}

The following result prepares the application of the global Leray-Schauder theorem in Proposition \ref{Leray:Schauder}.

\begin{proposition} \label{image:chain:connected}
If $M\subset\Delta_\rho$ is chain-connected, then $\Phi_{\alpha^*}(M)\cap\Delta_\rho$ is chain-connected.
\end{proposition}
\begin{proof}
Let $x,\tilde x\in M$ with $|x-\tilde x|_\infty\le\rho$, $y\in\Phi_{\alpha^*}(x)\cap\Delta_\rho$, and 
$\tilde y\in\Phi_{\alpha^*}(\tilde x)\cap\Delta_\rho$. Moreover, let $\eta\in\Phi(x)$.
Then there exists $\tilde\eta\in\Phi(\tilde x)$ such that $|\eta-\tilde\eta|_\infty\le(1+Lh)\rho$, and hence 
$z^*:=(\eta+\tilde\eta)/2\in\Phi_{\alpha^*}(x)\cap\Phi_{\alpha^*}(\tilde x)$. 

According to Lemma \ref{existence:of:grid:point}, there exists 
$\xi\in B_{\alpha^*}(\eta)_\infty\cap\Delta_\rho\subset\Phi_{\alpha^*}(x)\cap\Delta_\rho$ 
such that $|\xi-z^*|_\infty<\rho$. By the same argument, there exists 
$\tilde\xi\in B_{\alpha^*}(\tilde z)_\infty\cap\Delta_\rho\subset\Phi_{\alpha^*}(\tilde x)\cap\Delta_\rho$ 
such that $|\tilde\xi-z^*|_\infty<\rho$. In particular, $|\xi-\tilde\xi|_\infty\le|\xi-z^*|_\infty+|z^*-\tilde\xi|_\infty<2\rho$  
forces $|\xi-\tilde\xi|_\infty\le\rho$, because $\xi,\tilde\xi\in\Delta_\rho$.

By Lemma \ref{single:image:connected}, there exist $c\in\chain(y,\xi)$ and $\tilde c\in\chain(\tilde\xi,\tilde y)$
with $c\subset\Phi_{\alpha^*}(x)\cap\Delta_\rho$ and $\tilde c\subset\Phi_{\alpha^*}(\tilde x)\cap\Delta_\rho$. 
But then $\hat c:=\tilde c\circ(\xi,\tilde\xi)\circ c\in\chain(y,\tilde y)$
and $\hat c\subset(\Phi_{\alpha^*}(x)\cup\Phi_{\alpha^*}(\tilde x))\cap\Delta_\rho$, 
so that $(\Phi_{\alpha^*}(x)\cup\Phi_{\alpha^*}(\tilde x))\cap\Delta_\rho$ is chain-connected.

Now let $y_*\in\Phi_{\alpha^*}(M)\cap\Delta_\rho$ and $y^*\in\Phi_{\alpha^*}(M)\cap\Delta_\rho$ be arbitrary.
There exist points $x_*,x^*\in M$ and $\bar c=(c_n)_{n=0}^N\in\chain(x_*,x^*)$ with $\bar c\subset M$ such that
$y_*\in\Phi_{\alpha^*}(x_*)\cap\Delta_\rho$ and $y^*\in\Phi_{\alpha^*}(x^*)\cap\Delta_\rho$. Repeated application
of the above argument yields that $(\cup_{n=0}^N\Phi_{\alpha^*}(\bar c_n))\cap\Delta_\rho$ is chain-connected,
and, in particular, there exists a chain in $\Phi_{\alpha^*}(M)\cap\Delta_\rho$ connecting $y_*$ and $y^*$. 
\end{proof}

\subsection{Application of the global Leray-Schauder Theorem}

As the discrete chain-connected set $M\subset\Delta_\rho$ is not accessible for the global Leray-Schauder Theorem, we use
a modified version of $M$. 

\begin{lemma} \label{properties:M:hat}
For any $\gamma^*>0$ with $\rho/2<\gamma^*<\rho$, the set $\hat M := B_{\gamma^*}(M)_\infty$ is compact and satisfies
\[\dist_H(\partial\hat{M},\partial_\rho^0M)_\infty \le \gamma^*.\]
Moreover, it is strongly path-connected in the sense that for any $z,\tilde z\in\hat M$ there exists $\phi\in C([0,1],\R^d)$ 
with $\phi(0)=z$, $\phi(1)=\tilde{z}$, and $\phi((0,1))\subset\interior\hat{M}$.
\end{lemma}
\begin{proof}
The set $\hat M$ is closed and bounded and hence compact. 

\medskip

\emph{Distance between boundaries.} 
For any $z\in B_\rho(\interior_\rho M)$, there exists some $x\in M$ such that $z\in B_{\rho/2}(x)\subset\hat M$,
so that $B_\rho(\interior_\rho M) \subset \hat M$.
Now let $z\in\partial\hat M$. There exists $x\in M$ with $|z-x|_\infty\le\gamma^*<\rho$, and by the above,
we have $x\notin\interior_\rho M$, so that $x\in\partial_\rho^0 M$, and hence 
\[\dist(\partial\hat{M},\partial_\rho^0M)_\infty \le \gamma^*.\]

Now let $x\in\partial_\rho^0M$. There exists $z\in M^c\cap\Delta_\rho$ with $|z-x|_\infty=\rho$ and $|z-\tilde x|_\infty\ge\rho$
for all $\tilde x\in M$. Hence $\interior B_{\rho-\gamma^*}(z)_\infty \subset \hat{M}^c$, and since $B_{\gamma^*}(x)\subset\hat{M}$,
we have $\tilde z:=\frac{\gamma^*}{\rho}z + \frac{\rho-\gamma^*}{\rho}x\in\partial\hat{M}$ and $|\tilde z-x|_\infty=\gamma^*$. Thus
\[\dist(\partial_\rho^0M,\partial\hat{M})_\infty \le \gamma^*.\]

\emph{Strong path-connectedness.} Let $z,\tilde z\in\hat M$. By definition of $\hat M$, there exist
$x,\tilde x\in M$ with $|z-x|_\infty\le\gamma^*$ and $|\tilde z-\tilde x|_\infty\le\gamma^*$.
But then the paths $\phi\in C([0,1],\R^d)$ and 
$\tilde\phi\in C([0,1],\R^d)$ given by
$\phi(\lambda):=\lambda x+(1-\lambda) z$ and $\tilde\phi(\lambda):=\lambda\tilde z+(1-\lambda) \tilde x$
satisfy $\phi([0,1])\subset B_{\gamma^*}(x)_\infty\subset \hat M$, $\phi(0)=z$, $\phi(1)=x$, 
$\tilde\phi([0,1])\subset B_{\gamma^*}(\tilde x)_\infty\subset \hat M$,
$\tilde\phi(0)=\tilde x$, $\tilde\phi(1)=\tilde z$ and
$\phi((0,1)),\tilde\phi((0,1))\subset\interior \hat M$.

As $M$ is chain-connected, there exists a chain $(c_n)_{n=0}^N\subset M$ with $c_0=x$ and $c_N=\tilde x$.
Then $\phi_n\in C([0,1],\R^d)$  given by $\phi_n(\lambda):=\lambda c_{n+1} + (1-\lambda)c_n$
satisfies $\phi_n([0,1])\subset\interior\hat{M}$,
$\phi_n(0)=c_n$, and $\phi_n(1)=c_{n+1}$, and hence the path
\[\hat\phi:=\tilde\phi\circ\phi_{N-1}\circ\phi_{N-2}\circ\ldots\circ\phi_1\circ\phi_0\circ\phi\in C([0,1],\R^d)\]
has the desired properties.
\end{proof}

The following proposition allows us to work essentially with the boundary of the 
individual images of the Euler scheme.

\begin{proposition} \label{Leray:Schauder}
Let $y\in M^c\cap\Delta_\rho$ and $z\in\interior_\rho M$ be such that $\xi\in\Phi_{\alpha^*,\beta^*}^\partial(y)$ and $\xi\in\Phi_{\alpha^*,\beta^*}(z)$.
Then there exists some $\tilde w\in\partial_\rho^0 M$ such that $\xi\in\Phi_{\alpha^*}(\tilde w)$ and 
$\dist(\xi,\Phi_{\alpha^*}^\partial(\tilde w))_\infty\le\hat\kappa(h,\rho,\beta^*)$ with
\[\hat\kappa(h,\rho,\beta^*):=\frac{2+2Lh}{1-Lh}\alpha^* + \frac{3+Lh}{1-Lh}\beta^* + (1+Lh)\dist(y,M)_\infty.\]
\end{proposition}
\begin{proof}
By Proposition \ref{solvable} there exists $\tilde y\in\Phi^{-1}(\xi)$ such that
\begin{align*}
|\tilde y-y|_\infty \le \frac{\dist(\xi,\Phi(y))_\infty}{1-Lh} \le \frac{\alpha^*+\beta^*}{1-Lh} < \frac{\rho}{1-Lh} \le \frac{4}{3}\rho.
\end{align*}

\medskip

\emph{Case 1: $\dist(\tilde y,M)_\infty\le\rho/2$.} Then there exists
$\tilde w\in\partial_\rho^0M$ with $|\tilde y-\tilde w|_\infty\le\rho/2$, and $\xi\in\Phi_{\alpha^*}(\tilde w)$ holds by Lipschitz continuity.
Moreover,
\begin{align*}
&\dist(\xi,\Phi_{\alpha^*}^\partial(\tilde w))_\infty \le \dist(\xi,\Phi_{\alpha^*}^\partial(y))_\infty 
+ \dist(\Phi_{\alpha^*}^\partial(y),\Phi_{\alpha^*}^\partial(\tilde w))_\infty\\
&\le \beta^* + (1+Lh)|y-\tilde w|_\infty \le \beta^* + (1+Lh)(\frac{\alpha^*+\beta^*}{1-Lh}+\rho/2) \le 2\frac{\alpha^*+\beta^*}{1-Lh}.
\end{align*}

\medskip

\emph{Case 2: $\dist(\tilde y,M)_\infty>\rho/2$.} Fix $\gamma^*>0$ with $\rho/2<\gamma^*<\dist(\tilde y,M)_\infty$,
so that $\tilde y\notin \hat M$. There exists $\tilde z\in\Phi^{-1}(\xi)$ such that
\[|\tilde z-z|_\infty \le \frac{\dist(\xi,\Phi(z))_\infty}{1-Lh} \le \frac{\alpha^*+\beta^*}{1-Lh} \le \frac{\rho}{1-Lh} \le \frac{4}{3}\rho,\]
so that $\tilde z\in\interior\hat M$.
By Proposition \ref{connected}, there exists $\phi\in C([0,1],\R^d)$ such that $\phi([0,1])\subset\Phi^{-1}(\xi)$, 
$\phi(0)=\tilde z$, and $\phi(1)=\tilde y$. Let $\lambda^*:=\sup\{\lambda\in[0,1]: \phi(\lambda)\in\hat M\}$. 
By the above and by continuity of $\phi(\cdot)$, we have $0<\lambda^*<1$.

By continuity of $\phi(\cdot)$ and compactness of $\hat M$, we have $x^*:=\phi(\lambda^*)\in\partial\hat M$.
Moreover, $\phi(\lambda)\in\hat M^c$ for all $\lambda\in(\lambda^*,1]$. 
Take $\tilde x\in\Proj(\tilde y,\hat M)_\infty \subset\partial\hat M$. Then the path defined by $\tilde\phi(\lambda):=\lambda\tilde x+(1-\lambda)\tilde y$
is an element $\tilde\phi\in C([0,1],\R^d)$ satisfying $\tilde\phi(0)=\tilde y$, $\tilde\phi(1)=\tilde x$, and $\tilde\phi(\lambda)\notin\hat M$
for all $\lambda\in[0,1)$. As a consequence, the path $\psi:=\tilde\phi\circ\hat\phi$ with $\hat\phi(\lambda):=\phi(\lambda+(1-\lambda)\lambda^*)$ satisfies
$\psi(0)=x^*$, $\psi(1)=\tilde x$, and $\psi(\lambda)\notin\hat M$ for all $\lambda\in(0,1)$.

By Lemma \ref{properties:M:hat}, there exists some $\tilde\psi\in C([0,1],\R^d)$ such that 
$\tilde\psi(0)=x^*$, $\tilde\psi(1)=\tilde x$, and $\tilde\psi(\lambda)\in\interior\hat M$ for all $\lambda\in(0,1)$.
Let $\delta_{\hat M}(\cdot)$ denote the signed distance function of the set $\hat M$.
Then the mapping 
\[H(\lambda,\eta):=\delta_{\hat M}(\eta\psi(\lambda)+(1-\eta)\tilde\psi(\lambda))+\eta\]
satisfies the assumptions of the global Leray-Schauder Theorem \cite[Theorem 14C]{Zeidler:1} on every rectangle
$[1/n,1-1/n]\times[0,1]$ for $n\in\N$ with $n>1$. As a consequence, there exist compact connected sets
$C_n\subset[0,1]^2$ such that 
\[C_n\cap(\{1/n\}\times[0,1])\neq\emptyset\neq C_n\cap(\{1-1/n\}\times[0,1])\]
and 
\begin{equation} \label{local:1}
\delta_{\hat M}(\eta\psi(\lambda)+(1-\eta)\tilde\psi(\lambda))=0
\end{equation}
for all $(\lambda,\eta)\in C_n$.
Since $C_n\subset[0,1]^2$ is compact for all $n\in\N_1$, there exists a compact set $C\subset[0,1]^2$ such 
that $\lim_{n\rightarrow\infty}\dist_H(C_n,C)_\infty=0$ along a subsequence according to Blaschke's selection theorem \cite[Chapter 4]{Eggleston:58}. 
It follows that 
\[C\cap(\{0\}\times[0,1])\neq\emptyset\neq C\cap(\{1\}\times[0,1])\]
and \eqref{local:1} holds for all $(\lambda,\eta)\in C$.

The set $C$ is connected. Otherwise, there exist disjoint open sets $A_1$ and $A_2$ such that $C\subset A_1\cup A_2$.
Since $C_n$ are connected, there exist points $c_n\in C_n\subset[0,1]^2$ such that $c_n\notin A_1\cup A_2$. As 
$[0,1]^2\setminus(A_1\cup A_2)$ is compact, there exists $c\in[0,1]^2\setminus(A_1\cup A_2)$ such that $c_n\rightarrow c$ 
along a subsequence. By construction $c\in C$, which is a contradiction.

The set 
\[E:=\{\eta\psi(\lambda)+(1-\eta)\tilde\psi(\lambda): (\eta,\lambda)\in C\}\]
is compact and connected as a continuous image of the compact and connected set $C$. We have $x^*,\tilde x\in E$ and
$E\subset \partial\hat M$. 
Since $\xi\in\Phi(x^*)$ and 
\begin{align*}
\dist(\xi,\Phi_{\alpha^*}^\partial(\tilde x))_\infty 
&\le \dist(\xi,\Phi_{\alpha^*}^\partial(y))_\infty + \dist(\Phi_{\alpha^*}^\partial(y),\Phi_{\alpha^*}^\partial(\tilde x))_\infty\\
&\le \beta^*+(1+Lh)(|y-\tilde y|_\infty+|\tilde y-\tilde x|_\infty)\\
&\le \beta^* + (1+Lh)(2|y-\tilde y|_\infty + \dist(y,\hat M)_\infty)\\
&\le \beta^* + 2(1+Lh)\frac{\alpha^*+\beta^*}{1-Lh} + (1+Lh)(\dist(y,M)_\infty-\rho/2)\\
&\le \frac{1+3Lh}{1-Lh}\alpha^* + \frac{3+Lh}{1-Lh}\beta^* + (1+Lh)\dist(y,M)_\infty =: \kappa(h,\rho,\beta^*)
\end{align*}
the sets $D_1:=\Phi^{-1}(\xi)\cap E$ and 
\[D_2:=\{x\in E: \dist(\xi,\Phi_{\alpha^*}^\partial(x))_\infty\le\kappa(h,\rho,\beta^*)\}\]
are nonempty. Assume that $D_1\cap D_2=\emptyset$ and let $x\in D_1$. Since $x\notin D_2$, $\kappa(h,\rho,\beta^*)\ge3\alpha^*$, and $\Phi(\cdot)$
is Lipschitz, there exists $\epsilon>0$ such that $\xi\in\Phi(w)$ for every $w\in B_\epsilon(x)$.
Hence $D_1$ is open in $E$. On the other hand, $D_1$ is closed by construction and continuity of $\Phi(\cdot)$,
so that connectedness of $E$ implies $E=D_1$. But this contradicts $D_2\neq\emptyset$. As a consequence,
there exists $\hat w\in D_1\cap D_2$, and this particular element satisfies $\hat w\in\partial\hat M$, $\xi\in\Phi(\hat w)$, and
\[\xi\in B_{\kappa(h,\rho,\beta^*)}(\Phi_{\alpha^*}^\partial(\hat w))_\infty.\]
By Lemma \ref{properties:M:hat}, there exists $\tilde w\in\partial_\rho^0 M$ such that $|\tilde w-\hat w|_\infty \le \gamma^*$,
and Lipschitz continuity implies $\xi\in\Phi_{(1+Lh)\gamma^*}(\tilde w)$ and
\[\dist(\xi,\Phi_{\alpha^*}^\partial(\tilde w))_\infty\le\kappa(h,\rho,\beta^*)+(1+Lh)\gamma^*.\]
Since the above considerations are correct for any admissible $\gamma^*>\rho/2$ and $M$ contains only finitely many points, 
the statement of the proposition ensues.
\end{proof}

Now Propositions \ref{boundary}, \ref{layer}, and \ref{Leray:Schauder} enable the following boundary Euler algorithm.
A justification is given in Theorem \ref{summary:2}.

\bigskip

\noindent\fbox{\begin{minipage}{\textwidth}
Assume that $\partial_\rho^0M$ and $\partial_\rho^1M$ are known and that $M$ is chain-connected.
\begin{itemize}
\item [1.] Compute $\partial_\rho^{-1}M$ and $\partial_\rho^2M$.
\item [2.] Compute 
$S_0^0:=(\Phi_{\alpha^*,\beta^*}(\partial_\rho^0M)\cap B_{\hat\kappa(h,\rho,\beta^*)}(\Phi_{\alpha^*,\beta^*}^\partial(\partial_\rho^0M)))\cap\Delta_\rho$,\\
$S_0^{-1}:= \Phi_{\alpha^*,\beta^*}^\partial(\partial_\rho^{-1}M))\cap\Delta_\rho$, 
and $S_1:=(\Phi_{\alpha^*,\beta^*}^\partial(\partial_\rho^1M\cup\partial_\rho^2M))\cap\Delta_\rho$.

\item [3.] Compute $\partial_\rho^1(\Phi_{\alpha^*,\beta^*}(M))=\{x\in S_1: \dist(x,S_0^0\cup S_0^{-1})=\rho\}$.
\item [4.] Compute $\partial_\rho^0(\Phi_{\alpha^*,\beta^*}(M))=\{x\in S_0^0\cup S_0^{-1}: \dist(x,\partial_\rho^1(\Phi_{\alpha^*,\beta^*}(M)))=\rho\}$.
\end{itemize}
\end{minipage}}

\bigskip

A graphic explanation of this algorithm is given in Section \ref{section:topological:change}.

\begin{theorem} \label{summary:2}
The boundary Euler scheme is well-defined.
\end{theorem}
\begin{proof}
By Proposition \ref{boundary} we have 
\begin{equation} \label{tag:5}
\partial_\rho^0(\Phi_{\alpha^*,\beta^*}(M)) \subset S_0^0\cup S_0^{-1} \subset \Phi_{\alpha^*,\beta^*}(M),
\end{equation}
and by Proposition \ref{layer} we have $\partial_\rho^1(\Phi_{\alpha^*,\beta^*}(M)) \subset S_1$,
so that together with \eqref{tag:5}
\[\partial_\rho^1(\Phi_{\alpha^*,\beta^*}(M))\subset\{x\in S_1: \dist(x,S_0^0\cup S_0^{-1})=\rho\}\] 
ensues.
Since Proposition \ref{Leray:Schauder} ensures $S_1\setminus (S_0^0\cup S_0^{-1})\subset\Phi_{\alpha^*,\beta^*}(M)^c$,
it follows that
\[\partial_\rho^1(\Phi_{\alpha^*,\beta^*}(M))\supset\{x\in S_1: \dist(x,S_0^0\cup S_0^{-1})=\rho\}.\] 
Finally,
\[\partial_\rho^0(\Phi_{\alpha^*,\beta^*}(M))=\{x\in S_0^0\cup S_0^{-1}: \dist(x,\partial_\rho^1(\Phi_{\alpha^*,\beta^*}(M)))=\rho\}\]
follows from \eqref{tag:5}.
\end{proof}

\section{Numerical examples}

In the following, the boundary Euler scheme will be examined in several numerical 
tests. Its speed of convergence, its ability to cope with topological changes of the reachable set,
and its failure due to violation of the assumptions imposed in this paper are addressed in 
carefully chosen examples. 

The boundary Euler scheme and the classical Euler scheme are implemented in C++, and the hash
container class unordered\_set of the open source library boost is used to store the data
efficiently. The visualization is done in Matlab.

\subsection{Speed of convergence}

\begin{figure}[p]
\begin{center}
\includegraphics[scale=0.8]{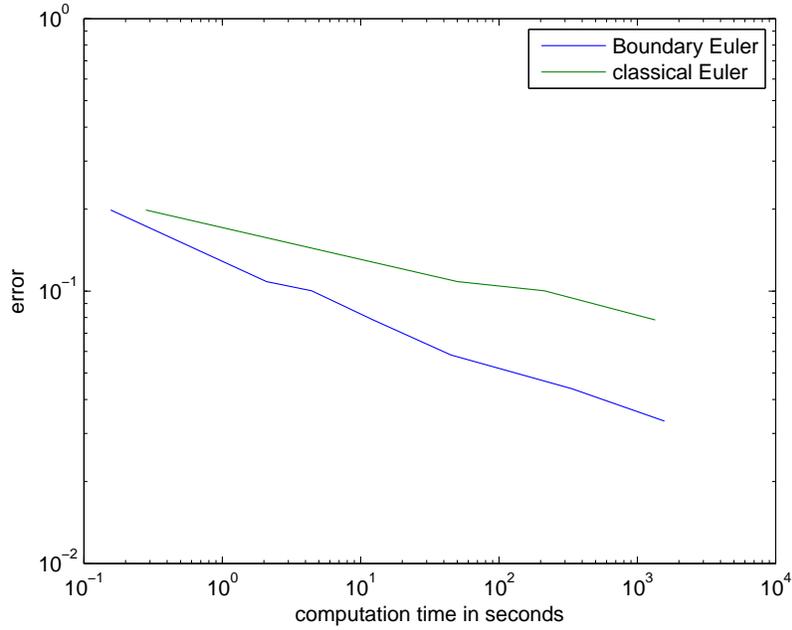}
\end{center}
\caption{Numerical errors of classical Euler scheme and boundary Euler scheme applied to \eqref{linear:odi} plotted against
computational cost. \label{error:figure}}
\end{figure}

\begin{table}[p]
\begin{center}
\begin{tabular}{c|c|c}
time Euler scheme [s] & time boundary Euler [s] & numerical error\\
\hline
0.821 & 0.156 & 0.1983\\
50.13 & 2.09 & 0.1083\\
213.73 & 4.43 & 0.1002\\
1350.25 & 12.355 & 0.0783\\
- & 45.006 & 0.0583\\
- & 336.79 & 0.0438\\
- & 1570.93 & 0.0333
\end{tabular}
\end{center}
\caption{Numerical errors of both schemes applied to \eqref{linear:odi} and corresponding
computational cost. \label{error:table}}
\end{table}

For measuring the speed of convergence, we consider the ordinary differential inclusion
\begin{equation} \label{linear:odi}
x'(t) \in x(t) + B_1(0)_\infty,\quad x(0)=0\in\R^2,
\end{equation}
because the behavior of the numerical error is realistic while the inclusion is
simple enough to admit the closed solution
\[\mc{R}(T,\{0\})=B_{e^t-1}(0)_\infty.\]
Moreover, the simple structure of the reachable set allows a reliable computation
of the numerical errors w.r.t.\ the Hausdorff distance. The numerical errors and the 
corresponding computation times (on an ordinary laptop) are listed in Table 
\ref{error:table} and visualized in Figure \ref{error:figure}.

As explained in \cite{Beyn:Rieger:07}, it is reasonable to
use the spatial discretization parameter $\rho=h^2$.
It is doubtful whether the notion of an order of convergence makes sense
in the set-valued context, because the performance of the schemes depends heavily
on the space dimension. In this particular example, however, 
the estimated rates of convergence in terms of the computational costs necessary to achieve 
a given precision are 0.108 for the classical Euler scheme and 0.192 for the Boundary
Euler, which makes a remarkable difference.

\subsection{Topological changes of the reachable set} \label{section:topological:change}

\begin{figure}[t]
\begin{center}
\includegraphics[scale=0.8]{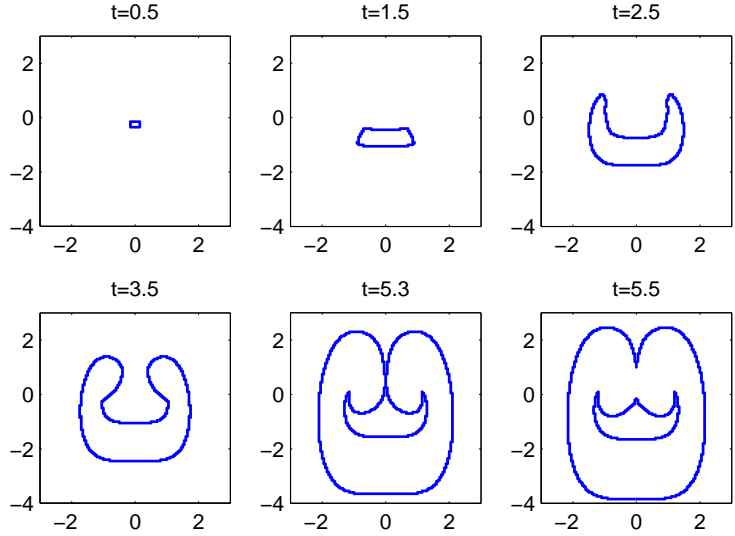}
\end{center}
\caption{The boundary Euler scheme applied to \eqref{Schnurrbart} and change of topology of the reachable set. \label{schnurrbart:figure}}
\end{figure}

\begin{figure}[p]
\begin{center}
\includegraphics[scale=0.8]{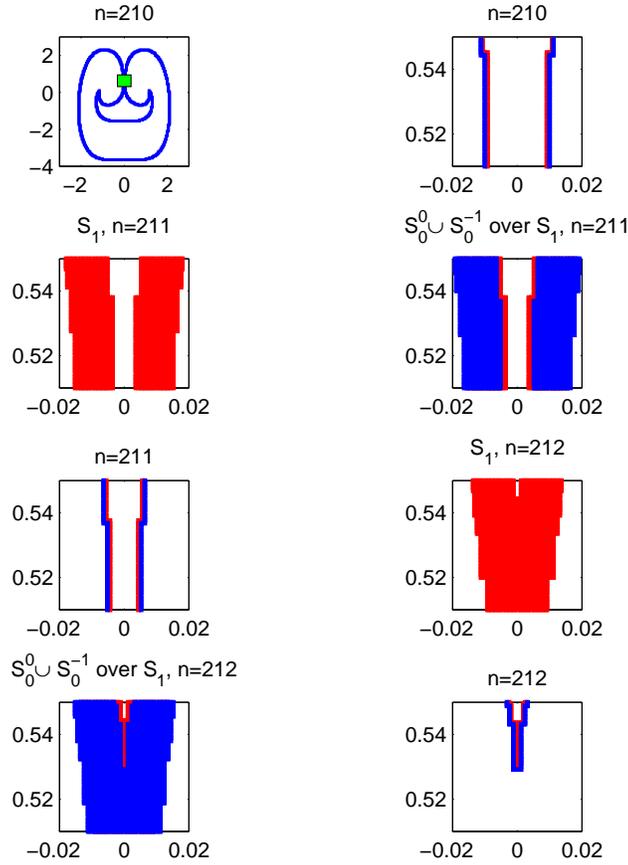}
\end{center}
\caption{A closeup of the critical time-step where the reachable set of the boundary Euler scheme applied to \eqref{Schnurrbart} 
changes its topology. The green rectangle in
the first plot indicates the location of the magnified region. Points generating or being part of the outer layer are colored red,
points generating or being part of the boundary of the reachable set are colored blue. \label{closeup:figure}}
\end{figure}

Unlike sets transported by ordinary differential equations, sets evolving under differential inclusions can
change their topology. Consider the nonlinear ordinary differential inclusion
\begin{equation} \label{Schnurrbart}
\begin{pmatrix} x_1'\\x_2' \end{pmatrix} = \begin{pmatrix} x_1(1-|x_1|)-x_1x_2\\ x_1^4-1/2 \end{pmatrix} + B_{1/5}(0)_\infty,\quad 
x(0)=0\in\R^2.
\end{equation}
The reachable set $\mc{R}(T,\{0\})$ of inclusion \eqref{Schnurrbart} is simply connected
for $T\in[0,5.275]$, but not for $T=5.3$. The evolution of the reachable set for $h=0.025$ and $\rho=h^2$ 
is shown in Figure \ref{schnurrbart:figure}. 

A closeup of the critical
transition is depicted in Figure \ref{closeup:figure}, which needs some explanation. The small green box in the first plot indicates the 
location of the magnified spot, and the second shows the situation at this location after 210 iterations. The blue line is the 
boundary $\partial_\rho^0\tilde{\mc{R}}_h(5.275,\{0\})$ of the reachable set, and the red line is the layer 
$\partial_\rho^1\tilde{\mc{R}}_h(5.275,\{0\})$
of exterior points the algorithm needs to keep track of the topology of the set. 

\begin{figure}[t]
\begin{center}
\includegraphics[scale=0.8]{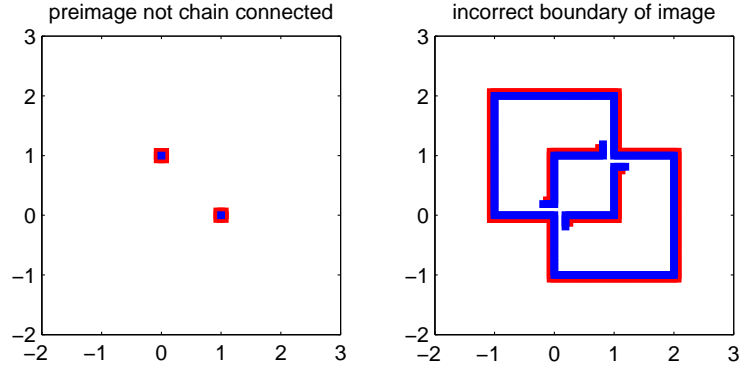}
\end{center}
\caption{Failure of the boundary Euler scheme for disconnected initial sets. \label{failure:figure:1}}
\end{figure}

The third plot shows the set $S_1$ (specified in the algorithm)
that contains the layer $\partial_\rho^1\tilde{\mc{R}}_h(5.3,\{0\})$ of exterior points of the set
$\tilde{\mc{R}}_h(5.3,\{0\})$ to be computed. In the fourth plot, the set $S_1$ is overlaid by $S_0^0\cup S_0^{-1}$,
deleting everything but $\partial_\rho^1\tilde{\mc{R}}_h(5.3,\{0\})$ from $S_1$, and in the fifth,
only those points of $S_0^0\cup S_0^{-1}$ are kept that have a neighbor in $\partial_\rho^1\tilde{\mc{R}}_h(5.3,\{0\})$,
so that $S_0^0\cup S_0^{-1}$ is reduced to $\partial_\rho^0\tilde{\mc{R}}_h(5.3,\{0\})$.

In the following plots, this procedure is repeated, but the seventh plot shows that in this step $S_0^0\cup S_0^{-1}$
erases a large portion of $S_1$, so that no new boundary is generated in that place and the change of topology of the 
exact reachable set is reproduced by the discrete approximation. 

\medskip

The opposite effect -- the closing of a hole in the reachable set -- can be observed in the simple example
\begin{equation} \label{simple:inclusion}
x'(t) \in B_1(0)_\infty,\quad x(0)\in X_0=\{x\in\R^2: 1\le|x|_\infty\le 2\}
\end{equation}
that is not displayed here, because the mechanism behind the change of topology is essentially the same as for 
inclusion \eqref{Schnurrbart}.

\begin{figure} [t]
\begin{center}
\includegraphics[scale=0.8]{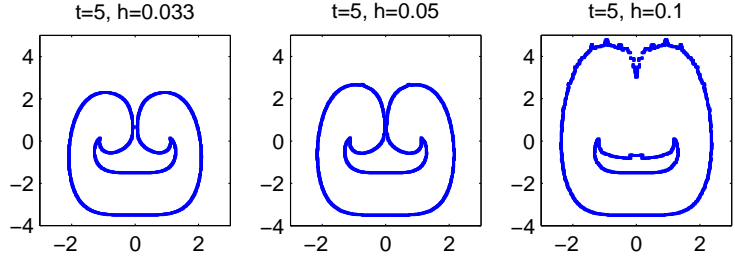}
\end{center}
\caption{Failure of the boundary Euler scheme because of $Lh\ge 1/4$. \label{failure:figure:2}}
\end{figure}

\subsection{Failure due to violated assumptions}

The proof of the propriety of the boundary Euler scheme relies on the assumption that the initial set $B_{\alpha^*}(X_0)\cap\Delta_\rho$
is chain-connected. Applying one step of the boundary Euler to the simple differential inclusion \eqref{simple:inclusion} with 
disconnected initial set $X_0=\{(0,1),(1,0)\}$, $h=1/4$, and $\rho=1/16$ demonstrates that this assumption is not imposed for convenience, 
but relevant for practical computation. Figure \ref{failure:figure:1} shows that the boundary $\partial_\rho^0\tilde{\mc{R}}(1/4,X_0)$
of the reachable set at time $1/4$ is not correctly computed, because inner points of the reachable set are marked as boundary points.
Therefore, if the initial set is disconnected, the preliminary version of the boundary Euler scheme must be used, which outperforms the classical Euler 
scheme, but is slower than the boundary Euler scheme in its final form.

\medskip

Figure \ref{failure:figure:2} shows what may happen if the assumption $Lh<1/4$ is violated.
In example \eqref{Schnurrbart}, the boundaries of the discrete reachable 
sets do not only become more and more inaccurate, but literally fall apart when $h$ is increased. 
This effect is known for the classical Euler scheme, but it is 
only ungainly and not harmful there, because the discrete reachable sets still approximate the original reachable sets
with prescribed accuracy. Since the boundary Euler is supposed to compute a real boundary, the effect 
is not tolerable in this setting.

%

%

\bibliographystyle{plain}
\bibliography{standard}

\end{document}